\documentclass{amsart}

\ifx\MyBeamer\undefined
\usepackage[breaklinks]{hyperref}
\usepackage[a4paper,margin=1in,includefoot]{geometry}
\else
\def\NoSVN{y}
\if\MyBeamer t
\usepackage[breaklinks]{hyperref}
\usepackage{beamerarticle}
\usepackage[a4paper,margin=1in,includefoot]{geometry}
\fi\fi

\usepackage{amsmath}
\usepackage{amssymb}
\usepackage{mathrsfs}
\usepackage{amsthm}
\usepackage[utf8]{inputenc}
\usepackage[T1]{fontenc}
\ifx\NoSVN\undefined
\usepackage{svninfo}
\fi
\usepackage{prettyref}
\ifx\FrenchText\undefined
\usepackage[british]{babel}
\else
\usepackage[british,french]{babel}
\AtBeginDocument{%
\catcode`\:=12%
\catcode`\!=12%
}
\fi

\makeatletter

\newcommand{\fref}[1]{\prettyref{#1}}
\newrefformat{ax}{(\ref{#1})}
\newrefformat{item}{\ref{#1}}
\newrefformat{apx}{Appendix~\ref{#1}}
\newrefformat{cha}{Chapter~\ref{#1}}
\newrefformat{sec}{Section~\ref{#1}}

\ifx\MyBeamer\undefined

\ifx\thmnum\undefined
\newtheorem{thmnum}{}[section]
\fi

\newcommand{\mynewthm}[3][]{%
  \newtheorem{#2}[thmnum]{#3}
  \newtheorem*{#2*}{#3}%
  \newrefformat{#2}{#3~\ref{##1}}%
}

\newtheorem*{clm}{Claim}
\newenvironment{clmprf}{%
  \begin{proof}[Proof of claim]%
  }{\end{proof}}

\else

\let\xxx=\frametitle
\def\frametitle#1{%
  \xxx{%
    \setbeamercolor*{math text}{use={titlelike,my math text},fg=titlelike.fg!80!my math text.fg}%
    #1}%
  \setbeamercolor{math text}{use=my math text,fg=my math text.fg}%
}

\newcommand{\beamerenv}[3]{%
\newenvironment<>{#1}%
{%
  \setbeamercolor{temp}{structure}%
  \setbeamercolor{structure}{fg=#2}%
  \setbeamercolor{block body}{use=structure,bg=structure.fg!5!white}%
  \begin{#3}%
}%
{\end{#3}\setbeamercolor{structure}{temp}}}

\newcommand{\mynewthm}[3][green!50!black]{%
  \newtheorem*{#2x}{#3}%
  \beamerenv{#2}{#1}{#2x}%
}

\beamerenv{other}{violet}{block}

\fi

\ifx\FrenchText\undefined
\newcommand{\myiffrench}[2]{#2}
\else
\newcommand{\myiffrench}[2]{\iflanguage{french}{#1}{#2}}
\fi

\theoremstyle{plain}
\mynewthm[red]{thm}{\myiffrench{Théorème}{Theorem}}
\mynewthm{prp}{Proposition}
\mynewthm[purple]{lem}{\myiffrench{Lemme}{Lemma}}
\mynewthm[orange]{fct}{\myiffrench{Fait}{Fact}}
\mynewthm[purple!50!white]{cor}{\myiffrench{Corollaire}{Corollary}}

\theoremstyle{definition}
\mynewthm[green!80!black]{dfn}{\myiffrench{Définition}{Definition}}
\mynewthm{hyp}{\myiffrench{Hypothèse}{Hypothesis}}
\mynewthm{conv}{Convention}
\mynewthm{conj}{Conjecture}
\mynewthm{ntn}{Notation}
\mynewthm{cst}{Construction}

\theoremstyle{remark}
\mynewthm[yellow!90!black]{rmk}{\myiffrench{Remarque}{Remark}}
\mynewthm{qst}{Question}
\mynewthm{exc}{\myiffrench{Exercice}{Exercise}}
\mynewthm{exm}{\myiffrench{Exemple}{Example}}

\ifx\MyBeamer\undefined

\newcommand{\myenumlabel}[1]{\textnormal{(\roman{#1})}}

\fi

\newcounter{cycprfcnt}
\newcommand{\cycprfpreamble}%
{%
  \setcounter{cycprfcnt}{1}
  \setlength{\itemindent}{0.5\leftmargin}%
  \setlength{\leftmargin}{0pt}%
  \newcommand{\cpcurr}{\myenumlabel{cycprfcnt}}%
  \newcommand{\cpnext}{\addtocounter{cycprfcnt}{1}\cpcurr}%
  \newcommand{\impnext}{\cpcurr{} $\Longrightarrow$ \cpnext.}%
}%

{\begin{list}{\impnext}%
  {\cycprfpreamble}}%
{\qedhere\end{list}}%

\newenvironment{cycprf*}%
{\begin{list}{\impnext}%
  {\cycprfpreamble}}%
{\end{list}}%

\def\indsym#1#2{%
  \setbox0=\hbox{$\m@th#1x$}%
  \kern\wd0%
  \hbox to 0pt{\hss$\m@th#1\mid$\hbox to 0pt{$\m@th#1^{#2}$\hss}\hss}%
  \lower.9\ht0\hbox to 0pt{\hss$\m@th#1\smile$\hss}%
  \kern\wd0}

\def\nindsym#1#2{%
  \setbox0=\hbox{$\m@th#1x$}%
  \kern\wd0%
  \hbox to 0pt{\hss$\m@th#1\not$\kern1.4\wd0\hss}
  \hbox to 0pt{\hss$\m@th#1\mid$\hbox to 0pt{$\m@th#1^{#2}$\hss}\hss}%
  \lower.9\ht0\hbox to 0pt{\hss$\m@th#1\smile$\hss}%
  \kern\wd0}

\def\dotminussym#1#2{%
  \setbox0=\hbox{$\m@th#1-$}%
  \kern.5\wd0%
  \hbox to 0pt{\hss\hbox{$\m@th#1-$}\hss}%
  \raise.6\ht0\hbox to 0pt{\hss$\m@th#1.$\hss}%
  \kern.5\wd0}

\renewcommand{\setminus}{\smallsetminus}

\newcommand{\rest}{{\restriction}}

\newcommand{\sfrac}[2]{\hbox{$\frac{#1}{#2}$}}
\newcommand{\half}[1][1]{\sfrac{#1}{2}}

\DeclareMathOperator{\id}{id}

\DeclareMathOperator{\Iso}{Iso}

\DeclareMathOperator{\img}{img}

\newcommand{\bG}{\mathbf{G}}

\newcommand{\bR}{\mathbf{R}}

\newcommand{\bU}{\mathbf{U}}

\makeatother

\DeclareMathOperator{\AEx}{AE}
\DeclareMathOperator{\Lip}{Lip}

\begin{document}

\title{The linear isometry group of the Gurarij space is universal}

\author{Itaï \textsc{Ben Yaacov}}

\address{Itaï \textsc{Ben Yaacov} \\
  Université Claude Bernard -- Lyon 1 \\
  Institut Camille Jordan, CNRS UMR 5208 \\
  43 boulevard du 11 novembre 1918 \\
  69622 Villeurbanne Cedex \\
  France}

\urladdr{\url{http://math.univ-lyon1.fr/~begnac/}}

\thanks{Research supported by the Institut Universitaire de France and ANR contract GruPoLoCo (ANR-11-JS01-008).}

\thanks{The author wishes to thank Vladimir V.\ \textsc{Uspenskij} for an inspiring discussion, and the organisers of «~Descriptive Set Theory in Paris 2012~» for facilitating this discussion.}

\svnInfo $Id: Katetov.tex 1382 2012-07-31 17:33:38Z begnac $
\thanks{\textit{Revision} {\svnInfoRevision} \textit{of} \today}

\keywords{Gurarij space ; Katětov function ; Arens-Eells space ; universal Polish group}
\subjclass[2010]{46B99 ; 22A05 ; 54D35}

\begin{abstract}
  We give a construction of the Gurarij space, analogous to Katětov's construction of the Urysohn space.
  The adaptation of Katětov's technique uses a generalisation of the Arens-Eells enveloping space to metric space with a distinguished normed subspace.

  This allows us to give a positive answer to a question of Uspenskij, whether the linear isometry group of the Gurarij space is a universal Polish group.
\end{abstract}

\maketitle

\section*{Introduction}

Let $\Iso(X)$ denote the isometry group of a metric space $X$, which we equip with the topology of point-wise convergence.
For complete separable $X$, this makes of $\Iso(X)$ a Polish group.
We let $\Iso_L(E)$ denote the linear isometry group of a normed space $E$ (throughout this paper, over the reals).
This is a closed subgroup of $\Iso(E)$, and is therefore Polish as well when $E$ is a separable Banach space.

It was shown by Uspenskij \cite{Uspenskij:UrysohnIsometryGroup} that the isometry group of the Urysohn space $\bU$ is a \emph{universal Polish group}, namely, that any other Polish group is homeomorphic to a (necessarily closed) subgroup of $\Iso(\bU)$, following a construction of $\bU$ due to Katětov \cite{Katetov:UniversalMetricSpaces}.
The Gurarij space $\bG$ (see \fref{dfn:Gurarij} below, as well as \cite{Gurarij:UniversalPlacement,Lusky:UniqueGurarij}) is, in \emph{some ways}, the analogue of the Urysohn space in the category of Banach spaces.
Either one is the unique separable, universal and approximately ultra-homogeneous object in its respective category, or equivalently, either one is the (necessarily unique) \emph{Fraïssé limit} (see \cite{BenYaacov:MetricFraisse}) of the finitely generated objects in its respective category.
This raises a natural question, put to the author by Uspenskij, namely whether $\Iso_L(\bG)$ is a universal Polish group as well.
Before discussing this question, let us briefly recall Katětov's construction and Uspenskij's argument for the Urysohn space.

Let $X$ be a metric space.
We say that a real-valued function $\xi$ on $X$ is \emph{Katětov} if $\xi(x) \leq \xi(y) + d(x,y)$ and $d(x,y) \leq \xi(x) + \xi(y)$ for all $x,y \in X$ (equivalently, $\xi(x) \geq | \xi(y) - d(x,y) |$ for all $x,y \in X$) (equivalently, if $\xi$ is the distance function from points in $X$ to some fixed point in a metric extension of $X$).
Let $K(X)$ denote the set of Katětov functions on $X$.
Equipped with the supremum distance, $K(X)$ is a complete metric space, endowed with a natural isometric embedding $X \hookrightarrow K(X)$, sending $x \mapsto d(x,\cdot)$.
This construction is functorial, in the sense that an isometric embedding $Y \hookrightarrow X$ gives rise to a natural isometric embedding $K(Y) \hookrightarrow K(X)$, where $\xi$ goes to $\hat \xi(x) = \inf_{y \in Y} d(x,y) + \xi(y)$, and everything composes and commutes as one would expect.
This functoriality gives rise to a natural embedding $\Iso(X) \hookrightarrow \Iso\bigl( K(X) \bigr)$, where the image of $\varphi \in \Iso(X)$ (call it $K(\varphi)$) extends $\varphi$.

There are two issues with $K(X)$ which make it less useful than one might hope.
First, even when $X$ is separable, $K(X)$ need not be separable.
More generally, letting $w(X)$ denote the weight of $X$ (namely the minimal size of a base), we only have $w(K(X)) \leq |K(X)| \leq 2^{w(X)}$, and equality may occur.
Second, since a point in $K(X)$ may ``depend'' on infinitely many points in $X$, the embedding $\Iso(X) \hookrightarrow \Iso\bigl( K(X) \bigr)$ need not be continuous.
The solution to both problems is to restrict our consideration to the Katětov functions which essentially only depend on finitely many elements, in the following sense.
When $Y \subseteq X$, let us identify $K(Y)$ with its image in $K(X)$, and define
\begin{gather}
  \label{eq:KatetovZero}
  K_0(X) = \overline{\bigcup_{Y \subseteq X \text{ finite}} K(Y)} \subseteq K(X).
\end{gather}
Then we have $X \hookrightarrow K_0(X)$ and this is functorial as for $K(X)$, yielding an embedding $\Iso(X) \hookrightarrow \Iso\bigl( K_0(X) \bigr)$.
Moreover, $w(K_0(X)) = w(X)$, and the embedding of topological groups above is continuous.
We now define $X_0 = X$, $X_{n+1} = K_0(X_n)$ and $X_\omega = \widehat{\bigcup X_n}$.
If $X$ is separable then so is $X_\omega$, and the latter is a Urysohn space.
At the same time we obtain continuous, and in fact homeomorphic, embeddings $\Iso(X) \subseteq \Iso(X_n)$ and therefore $\Iso(X) \subseteq \Iso(X_\omega) = \Iso(\bU)$, whence the desired result.
We refer the reader to \cite[Section~3]{Uspenskij:SubgroupsOfMinimalTopologicalGroups} for more details.

It turns out that the same strategy works for the Gurarij space as well.
In \fref{sec:ConvexKatetov} we identify the space of isomorphism types of one point extensions of a Banach space $E$ (in model-theoretic terminology we would speak of \emph{quantifier-free $1$-types} over $E$) as the space $K_C(E)$ of \emph{convex} Katětov functions.
The passage from $K$ to $K_0$ carries over essentially unchanged, in that $K_{C,0}(E) = K_C(E) \cap K_0(E)$ will do.
The main technical difficulty lies in the fact that, while $K(X)$ is a metric space extension of $X$, $K_C(E)$ is \emph{not} a Banach space extension of $E$.
In \fref{sec:ArensEells} we show that $K_C(E)$ embeds canonically in a Banach space extension of $E$.
In fact we prove a stronger result, generalising the Arens-Eells construction from pointed metric spaces to metric spaces over a normed space.
We conclude in \fref{sec:Main}, constructing, for a separable Banach space $E$, a chain of separable Banach space extensions as above, with limit $E_\omega$, and show that $E_\omega \cong \bG$, with only marginally more complexity than for pure metric spaces.

\section{Convex Katětov functions}
\label{sec:ConvexKatetov}

\begin{dfn}
  \label{dfn:ConvexKatetov}
  Let $E$ be a normed space, $X \subseteq E$ convex.
  We define $K_C(X) \subseteq K(X)$ to consist of all convex Katětov functions.
  We also define $K_{C,0}(X) = K_0(X) \cap K_C(X)$, where $K_0$ is as per \fref{eq:KatetovZero}.
\end{dfn}

We start by observing that $K_C(E)$ is the space of isomorphism types of one-point extensions of $E$.

\begin{lem}
  \label{lem:ConvexKatetov}
  Given $\xi \in K_C(E)$, let $E(x) = E \oplus \bR x$, and for $\alpha x - a \in E(x)$ define
  \begin{gather*}
    \|\alpha x - a\|^\xi =
    \begin{cases}
      |\alpha| \xi(a/\alpha) & \alpha \neq 0 \\
      \|a\| & \alpha = 0.
    \end{cases}
  \end{gather*}
  The map $\xi \mapsto \|{\cdot}\|^\xi$ is a bijection between $K_C(E)$ and semi-norms on $E(x)$ extending $\|{\cdot}\|_E$, whose inverse sends $\|{\cdot}\|$ to $\|x-\cdot\|$.
\end{lem}
\begin{proof}
  All there is to show is that if $\xi \in K_C(E)$ then $\|{\cdot}\|^\xi$ is indeed a semi-norm, and for this it will suffice to show that $\|\alpha x - a\|^\xi + \|\beta x - b\|^\xi \geq
  \|(\alpha+\beta)x - (a+b)\|^\xi$.
  We consider several cases:
  \begin{enumerate}
  \item If both $\alpha$ and $\beta$ are zero then there is nothing to show.
  \item If $\alpha$ and $\beta$ are non zero with equal sign, say $0 < \alpha,\beta$ and $\alpha + \beta = 1$, then we use convexity: $\alpha \xi(a/\alpha) + \beta(b/\beta) \geq \xi(\alpha a/\alpha + \beta b/\beta) = \xi(a+b)$.
  \item If $\alpha = -\beta$, say $\alpha = 1$, then this is $\xi(a) + \xi(-b) \geq \|a+b\| = d(a,-b)$, which follows from $\xi$ being Katětov.
  \item The last case is when $\alpha$ and $\beta$ have distinct signs and absolute values, say $\beta \leq 0 < \alpha$, and we may assume that $\alpha + \beta = 1$.
    Then indeed, using the hypothesis that $\xi$ is Katětov, if $\beta \neq 0$ then
    \begin{align*}
      \alpha \xi(a/\alpha) - \beta \xi(b/\beta)
      &
      = \xi(a/\alpha) - \beta \bigl[ \xi(a/\alpha) + \xi(b/\beta) \bigr]
      \\ &
      \geq \xi(a/\alpha) - \beta \|a/\alpha - b/\beta\|
      \\ &
      = \xi(a/\alpha) + \|a/\alpha - (a+b)\|
      \geq \xi(a+b).
    \end{align*}
    When $\beta = 0$ (and $\alpha = 1$), what we need to show is just the last inequality above, namely $\xi(a) + \|b\| \geq \xi(a+b)$.
    \qedhere
  \end{enumerate}
\end{proof}

From a model theoretic point of view, set $K_C(E)$ may therefore be identified with the space of quantifier-free $1$-types over $E$ (see \cite{BenYaacov-Henson:Gurarij}).
In addition, the notions of convex and Katětov functions are compatible in the following sense.

\begin{lem}
  \label{lem:ConvexKatetovCompatible}
  Let $E$ be a normed space, $X \subseteq Y \subseteq E$ convex.
  Then
  \begin{enumerate}
  \item The inclusion $K(X) \subseteq K(Y)$ restricts to $K_C(X) \subseteq K_C(Y)$ (in other words, the natural Katětov extension of a convex Katětov functions is convex as well).
  \item For $\xi \in K(X)$ we define the generated convex function $\xi^C\colon X \rightarrow \bR$ to be the greatest convex function lying below $\xi$, namely
    \begin{gather*}
      \xi^C(x) = \inf_{n ; \, \bar y \in A(x,n)} \sum \xi(y_k)/n, \qquad \text{ where } A(x,n) = \bigl\{ \bar y \in X^n\colon x = \sum y_i/n \bigr\}.
    \end{gather*}
    Then $\xi^C$ is Katětov as well.
  \end{enumerate}
  Moreover, if $X \subseteq Y \subseteq E$ are convex and $\xi \in K(X)$, then passing from $\xi$ to $\xi^C$ (on $X$) and then extending to $Y$ gives the same result as first extending $\xi$ to $Y$ and then passing to $\xi^C$, so we may refer to $\xi^C \in K_C(Y)$ without ambiguity.
\end{lem}
\begin{proof}
  For the first item, if $x,y \in Y$ and $\xi \in K_C(X)$, then
  \begin{gather*}
    \half[ \xi(x) + \xi(y) ] = \inf_{x',y' \in X} \half[ \|x-x'\| + \xi(x') + \|y-y'\| + \xi(y') ] \geq \inf_{x',y' \in X} \| \half[x+y] - \half[x'+y'] \| + \xi( \half[x'+y'] ) \geq \xi(\half[x+y]).
  \end{gather*}
  For the second item, we just take convex combinations of the two inequalities defining the Katětov property.
  The moreover part is a direct calculation.
\end{proof}

Our next step is to show that the intersection $K_C(X) \cap K_0(X)$ is not ``too small''.

\begin{lem}
  \label{lem:ConvexKatetovZero}
  For every normed space $E$ and convex subset $X \subseteq E$, we have
  \begin{gather*}
    K_{C,0}(X) = \overline{\bigcup \bigl\{ K_C(Y)\colon \text{convex compact } Y \subseteq X \bigr\} } = \bigl\{ \xi^C\colon \xi \in K_0(X) \bigr\}.
  \end{gather*}
\end{lem}
\begin{proof}
  We argue that
  \begin{gather*}
    K_{C,0}(X) \supseteq \overline{\bigcup \bigl\{ K_C(Y)\colon \text{convex compact } Y \subseteq X \bigr\} } \supseteq \bigl\{ \xi^C\colon \xi \in K_0(X) \bigr\} \supseteq K_{C,0}(X).
  \end{gather*}
  For the first inclusion it suffices to observe that $K_0(X)$ contains $K(Y)$ for every compact $Y \subseteq X$.
  The second inclusion holds since $\xi \mapsto \xi^C$ is continuous and the convex hull of a finite set is compact.
  The third is immediate.
\end{proof}

Given two (isometric) extensions $E \subseteq F_i$, $i = 0,1$, let first $F_0 \oplus_E F_1$ denote their co-product over $E$ as vector spaces, namely $F_0 \oplus F_1$ divided by equality of the two copies of $E$.
Then there exists a greatest (compatible) norm on $F_0 \oplus_E F_1$, given by $\|a - b\| = \inf_{c \in E} \|a - c\| + \|c - b\|$.
With $F_0 = \bigl( E(x), \|{\cdot}\|^\xi \bigr)$ for some $\xi \in K_C(E)$ and $F_1 = F$ fixed, we obtain a canonical extension of $\|{\cdot}\|^\xi$ to $F(x) = E(x) \oplus_E F$, compatible with the embedding $K_C(E) \subseteq K_C(F)$.

\begin{lem}
  \label{lem:ConvexKatetovDistance}
  Let $E$ be a normed space and $\xi_i \in K_C(E)$, $i = 0,1$.
  Let also $r_0 = d(\xi_0,\xi_1)$ and $r_1 = \inf_{a \in E} \xi_0(a) + \xi_1(a)$.
  Then for every $r_0 \leq r \leq r_1$ there exists a semi-norm $\|{\cdot}\|_r$ on $E(x_0,x_1) = E \oplus \bR x_0 \oplus \bR x_1$ whose restriction to $E(x_i)$ is $\|{\cdot}\|^{\xi_i}$, such that $\|x_0-x_1\|_r = r$.
\end{lem}
\begin{proof}
  When $r = r_1$, just take the greatest norm on $E(x_0) \oplus_E E(x_1)$ compatible with $\|{\cdot}\|^{\xi_i}$ as above.
  When $r = r_0$, this is a special case of an amalgamation result of Henson (see \cite{BenYaacov-Henson:Gurarij}).
  Indeed, for $a \in \alpha x_0 + \beta x_1 \in E(x_0,x_1)$ define
  \begin{gather*}
    \| a + \alpha x_0 + \beta x_1 \|' = \inf_{b \in E, \gamma \in \bR} \| b + (\alpha + \gamma) x_0\|^{\xi_0} + \| a-b + (\beta - \gamma) x_1 \|^{\xi_1} + |\gamma| r_0.
  \end{gather*}
  This is clearly a semi-norm, and $\|x_0 - x_1\|' \leq r_0$.
  For all $a + \alpha x_0 \in E(x_0)$ and $b,\gamma$, we have, by choice of $r_0$,
  \begin{gather*}
    \|a + \alpha x_0\|^{\xi_0} \leq \| b + (\alpha + \gamma) x_0\|^{\xi_0} + \|a-b - \gamma x_1\|^{\xi_1} + |\gamma| r_0,
  \end{gather*}
  so $\|{\cdot}\| = \|{\cdot}\|^{\xi_0}$ on $E(x_0)$, and similarly $\|{\cdot}\| = \|{\cdot}\|^{\xi_1}$ on $E(x_1)$.
  It follows that $\|x_0 - x_1\|' = r_0$, as desired.

  Intermediate values can be obtained as convex combinations of the two extreme cases.
\end{proof}

\section{The Arens-Eells space over a normed space}
\label{sec:ArensEells}

\begin{ntn}
  Given two pointed metric spaces $(X,0)$ and $(Y,0)$, we let $\Lip_0(X,Y)$ denote the space of all Lipschitz functions $\theta\colon X \rightarrow Y$ which send $0 \mapsto 0$, and for $\theta \in \Lip_0(X,Y)$ (or for that matter, for any Lipschitz function $\theta$) we let $L(\theta)$ denote its Lipschitz constant.
  If $Y = \bR$, we omit it.
\end{ntn}

We recall the following facts from Weaver \cite[Chapter~2.2]{Weaver:LipschitzAlgebras}:

\begin{fct}
  \label{fct:ArensEells}
  Let $(X,0)$ be a pointed metric space.
  Then there exists a Banach space $\AEx(X)$, together with an isometric embedding $X \subseteq \AEx(X)$ sending $0 \mapsto 0$, called the \emph{Arens-Eells space} of $X$, having the following universal property: every $\theta \in \Lip_0(X,F)$, where $F$ is a Banach space, admits a unique continuous linear extension $\theta'\colon \AEx(X) \rightarrow F$, and this unique extension satisfies $\|\theta'\| \leq L(\theta)$.

  This universal property characterises the Arens-Eells space up to a unique isometric isomorphism.
  Its dual Banach space $\AEx(X)^*$ is canonically isometrically isomorphic to $\Lip_0(X)$, the isomorphism consisting of sending a linear functional to its restriction to $X$.
\end{fct}

We shall generalise this as follows:

\begin{dfn}
  Let $E$ be a fixed normed space.
  By a \emph{metric space over $E$} we mean a pair $(X,\varphi)$ where $X$ is a metric space, $\varphi\colon E\rightarrow X$ is isometric, and for each $x \in X$ the function $a \mapsto d(\varphi a,x)$ is convex.
  Most of the time $\varphi$ will just be an inclusion map, in which case it is replaced with $E$ or simply omitted.

  For a normed space $F$ we also define $\Lip_E(X,F)$ to consist of all Lipschitz functions $\theta\colon X \rightarrow F$ which are linear on $E$, and if $F = \bR$ then we omit it.
\end{dfn}

For the trivial normed space $0$, a metric space over $0$ is the same thing as a pointed metric space.
Also, every isometric inclusion of normed spaces $E \subseteq F$ renders $F$ metric over $E$.

\begin{thm}
  \label{thm:RelativeArensEells}
  Let $X$ be a metric space over a normed space $E$.
  Then there exists a Banach space $\AEx(X,E)$, together with an isometric embedding $X \subseteq \AEx(X,E)$ which is linear on $E$, having the following universal property: every $\theta \in \Lip_E(X,F)$, where $F$ is a Banach space, admits a unique continuous linear extension $\theta'\colon \AEx(X,E) \rightarrow F$, and this unique extension satisfies $\|\theta'\| \leq L(\theta)$.

  This universal property characterises $\AEx(X,E)$ up to a unique isometric isomorphism, and we shall call it the \emph{Arens-Eells space of $X$ over $E$}.
  Its dual $\AEx(X,E)^*$ is isometrically isomorphic to $\Lip_E(X)$ via restriction to $X$.
\end{thm}
\begin{proof}
  Let $\iota\colon E \rightarrow \AEx(X)$ denote the inclusion map, so that we may distinguish between algebraic operations in $E$ and in $\AEx(X)$.
  Let $F_0 \subseteq \AEx(X)$ be the closed subspace generated by all expressions of the form $\iota(a + b) - \iota a - \iota b$ for $a,b \in E$, observing that it also contains $\iota \alpha a - \alpha \iota a$ for $a \in E$ and $\alpha \in \bR$.
  Define $\AEx(X,E)$ as the completion of the quotient $\AEx(X)/F_0$, and let $\psi \colon X \rightarrow \AEx(X,E)$ be the natural map.
  Then $\psi$ is linear and $1$-Lipschitz, and both the universal property and characterisation of the dual hold, so all we need to show is that $\psi$ is isometric.

  Indeed, let $x_i \in X$, $i = 0,1$, and let $\xi_i = d(\cdot,x_i) \in K_C(E)$.
  Let $r_0 \leq r_1$ be as in \fref{lem:ConvexKatetovDistance}, and observe that by the triangle inequality $r_0 \leq d(x_0,x_1) \leq r_1$.
  Therefore, by \fref{lem:ConvexKatetovDistance} there exists a semi-norm $\|{\cdot}\|$ on $E(x_0,x_1)$ which induces the same distance on $E\cup\{x_0,x_1\}$ as $X$ does.
  By the Hahn-Banach Theorem, there exists a linear functional $\lambda \in E(x_0,x_1)^*$ which norms $x_0 - x_1$: $\|\lambda\| = 1$ and $\lambda(x_0 - x_1) = d(x_0,x_1)$.
  The restriction of $\lambda$ to $E\cup \{x_0,x_1\}$ is $1$-Lipschitz, and by \cite[Theorem~1.5.6]{Weaver:LipschitzAlgebras}, it extends to a $1$-Lipschitz function $\lambda'\colon X \rightarrow \bR$.
  Then $\lambda' \in \Lip_E(X) = \AEx(X,E)^* \subseteq \Lip_0(X)$ and $\|\lambda'\| = 1$, witnessing that $\|\psi x_0 - \psi x_1\| \geq d(x_0,x_1)$, completing the proof.
\end{proof}

Similarly, it is shown in \cite{Weaver:LipschitzAlgebras} that if $Y \subseteq X$ then $\AEx(Y)$ is naturally identified with the subspace of $\AEx(X)$ generated by $Y$, and the same can be deduced for metric spaces over a normed space $E$.

\begin{cor}
  \label{cor:EmbedKC}
  Let $E$ be a normed space and $E \subseteq X \subseteq K_C(E)$ (here we identify $E$ with its image in $K_C(E)$).
  Then there exists a Banach space $E[X]$ together with an isometric embedding $X \subseteq E[X]$ such that $E \subseteq E[X]$ is a Banach space extension, and the following universal property holds: every Lipschitz map $\theta\colon X \rightarrow F$, where $F$ is a Banach space, which is linear on $E$, admits a unique continuous linear extension $\theta'\colon E[X] \rightarrow F$ with $\|\theta'\| \leq L(\theta)$.

  In particular, $E[X]$ is uniquely determined by this universal property, and is generated as a Banach space by $X$.
\end{cor}
\begin{proof}
  Just take $E[X] = \AEx(X,E)$.
\end{proof}

\section{Main Theorem}
\label{sec:Main}

We recall (say from Lusky \cite{Lusky:UniqueGurarij}) the definition of the Gurarij space.

\begin{dfn}
  \label{dfn:Gurarij}
  A \emph{Gurarij space} is a separable Banach space $\bG$ having the property that for any $\varepsilon > 0$, finite-dimensional Banach space $E \subseteq F$, and isometric embedding $\psi\colon E \rightarrow \bG$, there is a linear embedding $\varphi\colon F \rightarrow \bG$ extending $\psi$ which is, in addition, \emph{$\varepsilon$-isometric}, namely $(1 - \varepsilon)\|x\| \leq \|\varphi x\| \leq (1 + \varepsilon)\|x\|$ for all $x \in F$.
\end{dfn}

\begin{fct}
  \label{fct:UniqueGurarij}
  The Gurarij space $\bG$ exists and is unique.
\end{fct}
\begin{proof}
  Existence (which also follows from \fref{lem:MainConstruction} below) is due to Gurarij \cite{Gurarij:UniversalPlacement}, as is almost isometric uniqueness.
  Isometric uniqueness is due to Lusky \cite{Lusky:UniqueGurarij}, with a recent plethora of more ``elementary'' (that is to say, in the author's view, more model-theoretic) proofs \cite{Kubis-Solecki:GurariiUniqueness,BenYaacov-Henson:Gurarij,BenYaacov:MetricFraisse}.
\end{proof}

\begin{dfn}
  \label{dfn:MainConstruction}
  For a Banach space $E$ we define
  \begin{itemize}
  \item $E' = E[ K_{C,0}(E) ]$ as per \fref{cor:EmbedKC},
  \item $E_0 = E$, $E_{n+1} = E_n'$, $E_\omega = \overline{ \bigcup E_n }$.
  \end{itemize}
\end{dfn}

In order to show that this construction yields a Gurarij space, it will be convenient to use the following, which is a special case of an unpublished result of Henson:

\begin{fct}[See also \cite{BenYaacov-Henson:Gurarij}]
  \label{fct:BanachSpaceDistance}
  Let $E$ and $F$ be normed spaces, let $\bar x \in E^k$ and $\bar y \in F^k$, and let
  \begin{gather*}
    r = \sup_{\sum |s_i| = 1} \left| \bigl\| \sum s_i x_i \bigr\|_E - \bigl\| \sum s_i y_i \bigr\|_F \right|.
  \end{gather*}
  Then $r$ is the least real number for which exists a normed space $E_1$ and isometric embeddings $E \subseteq E_1$, $F \subseteq E_1$, such that $\|x_i-y_i\| \leq r$ for all $i$.
\end{fct}
\begin{proof}
  Let $E_1 = E \oplus F$ as vector spaces, and consider the maximal function $\|{\cdot}\|'\colon E_1 \rightarrow \bR$ satisfying, for all $x \in E$, $y \in F$ and $\bar s \in \bR^n$:
  \begin{gather*}
    \left\| x + y + \sum s_i(x_i - y_i) \right\|' \leq \|x\|_E + \|y\|_F + r \sum |s_i|.
  \end{gather*}
  This is clearly a semi-norm on $E$, and $\|x_i - y_i\|' \leq r$.
  For $x \in E$ we have $\|x\|' \leq \|x\|_E$, while on the other hand, for any $\bar s$ we have by choice of $r$:
  \begin{align*}
    \|x\|_E
    &
    \leq \left\| x - \sum s_i x_i \right\|_E + \left\| \sum s_i x_i \right\|_E
    \\ &
    \leq \left\| x - \sum s_i x_i \right\|_E + \left\| \sum s_i y_i \right\|_F + r \sum |s_i|.
  \end{align*}
  It follows that $\|x\|' = \|x\|_E$, and similarly for $y \in F$, whence the desired amalgam.
  It is clear that $r$ is least.
\end{proof}

\begin{lem}
  \label{lem:MainConstruction}
  Let $E$ be a Banach space.
  Then $E'$ and $E_\omega$ are Banach space extensions of $E$ with $w(E) = w(E') = w(E_\omega)$, and $E_\omega$ satisfies the second property of \fref{dfn:Gurarij}.
  In particular, if $E$ is separable then $E_\omega$ is a Gurarij space.
\end{lem}
\begin{proof}
  Since $E[X]$ is generated over $E$ by the image of $X$, we have $w(E[X]) = w(E) + w(X)$, whence $w(E) = w(E') = w(E_\omega)$.
  For the Gurarij property, it will be enough to show that for every finite-dimensional normed space $F_1$, sub-space $F_0$ of co-dimension one, and $\varepsilon > 0$, there exists $\delta > 0$ such that every linear $\delta$-isometry $\varphi\colon F_0 \rightarrow E_\omega$ there exists a linear $\varepsilon$-isometry $\varphi\colon F_1 \rightarrow E_\omega$ extending $\psi$.
  Let $v$ generate $F_1$ over $F_0$, and say $\|v\| = 1$.

  Let us first assume that $\psi$ is isometric, with $\psi F_0 \subseteq \bigcup_n E_n$, and define $\xi \in K_C(F_0)$ by $\xi(x) = \|x-v\|$.
  Fix $R > 0$, let $X \subseteq F_0$ be the closed ball of radius $R$, and let $\xi' = \xi\rest_X \in K_C(X) \subseteq K_C(F_0)$.
  Since $X$ is compact, $\xi' \in K_{C,0}(F_0)$, and by construction there exists $u \in \bigcup_n E_n$ such that $\|\psi x-u\|_{E_\omega} = \|x-v\|_{F_1}$ for all $x \in X$.
  In particular, $\|u\| = 1$ as well, so for $x \in F_0 \setminus X$ we have $\bigl| \|\psi x - u\|_{E_\omega} - \|x - v\|_{F_1} \bigr| \leq 2 \leq \frac{2}{R - 1} \|x-v\|_{F_1}$.
  Thus the extension $\varphi\colon F_1 \rightarrow E_\omega$ given by $v \mapsto u$ is $\frac{2}{R-1}$-isometric, and we may take $R$ as large as we wish.

  Now consider the general case, and let $\bar x \subseteq F_0$ be a basis, say with $\|x_i\| = 1$.
  Then there are constants $C$ and $C'$ such that for every $x \in F_0$ and linear $\theta\colon F_0 \rightarrow F$, where $F$ extends $F_0$, we have $\|x\| \leq C\|x + v\|$ and $\|x - \theta x\| \leq C' \max_i \|x_i - \theta x_i\| \|x\|$.
  Let $\delta = \frac{\varepsilon}{6 C C' + 1 + \varepsilon}$, and let $\psi\colon F_0 \rightarrow E_\omega$ be a $\delta$-isometry.
  Then by \fref{fct:BanachSpaceDistance}, we may assume that $E_\omega$ and $F_1$ are embedded isometrically in some $F$, in which $\|x_i - \psi x_i\| \leq \delta$ for all $i$.
  We can now choose $\bar x' \in \bigcup_n E_n$ such that $\|x_i - x'_i\| < 2\delta$, and let $\theta\colon F_0 \rightarrow E_\omega$ be $x_i \mapsto x_i'$, so $\|x - \theta x\| \leq 2 \delta C C' \|x + v\|$ for all $x \in F_0$.
  By the special case above, applied to $\langle \bar x' \rangle \subseteq \langle \bar x',v \rangle$, there exists $u \in E_\omega$ such that $\bigl| \| y + u \| - \| y + v \| \bigr| < \delta \| y + v \|$ for every $y \in \langle \bar x' \rangle$.
  We define $\varphi\colon F_1 \rightarrow E_\omega$ to agree with $\id_{F_0} = \psi$ on $F_0$ and send $v \mapsto u$.
  Then for $x \in F_0$,
  \begin{align*}
    \bigl| \|\varphi (x+v)\| - \|x+v\| \bigr|
    &
    = \bigl| \|x + u\| - \|x+v\| \bigr|
    \\ &
    \leq 2\|x - \theta x\| + \bigl| \|\theta x + u\| - \|\theta x+v\| \bigr|
    \\ &
    \leq 2\|x - \theta x\| + \delta \|\theta x + v\|
    \\ &
    \leq 3\|x - \theta x\| + \delta \|x + v\|
    \\ &
    \leq \bigl[ 6 C C' + 1 \bigr] \delta \|x + v\| < \varepsilon \|x + v\|,
  \end{align*}
  as desired.
\end{proof}

We recall from Uspenskij \cite{Uspenskij:SubgroupsOfMinimalTopologicalGroups},

\begin{dfn}
  \label{dfn:gEmbedding}
  An isometric embedding of metric spaces $X \subseteq Y$ is said to be a \emph{$g$-embedding} if there exists a continuous homomorphism $\Theta\colon \Iso(X) \rightarrow \Iso(Y)$ such that $\Theta \varphi$ extends $\varphi$ for each $\varphi \in \Iso(X)$.
\end{dfn}

Notice that since the restriction map $\img \Theta \rightarrow \Iso(X)$, $\varphi \mapsto \varphi\rest_X$, is continuous as well, such a map $\Theta$ is necessarily a homeomorphic embedding.

\begin{fct}[Mazur-Ulam Theorem, {\cite[Theorem~1.3.5]{Fleming-Jamison:IsometriesOnFunctionSpaces}}]
  \label{fct:MazurUlam}
  For a normed space $E$, $\Iso(E)$ is the group of affine isometries.
  Stated equivalently, $\Iso_L(E) = \{\varphi \in \Iso(E)\colon \varphi 0 = 0\}$.
\end{fct}

It follows that an isometric embedding of Banach spaces $E \subseteq F$ is a $g$-embedding if and only if there exists a continuous homomorphism $\Theta\colon \Iso_L(E) \rightarrow \Iso_L(F)$ such that $\Theta\varphi$ extends $\varphi$ for all $\varphi \in \Iso_L(E)$.
Indeed, given $\Theta\colon \Iso(E) \rightarrow \Iso(F)$ we have $(\Theta\varphi) 0 = \varphi 0 = 0$, so $\Theta$ restricts to $\Theta'\colon \Iso_L(E) \rightarrow \Iso_L(F)$, and conversely, given $\Theta\colon \Iso_L(E) \rightarrow \Iso_L(Y)$ we can define $\Theta'\colon \Iso(E) \rightarrow \Iso(F)$ by $\Theta'\varphi = \Theta(\varphi - \varphi 0) + \varphi 0$.

\begin{lem}
  \label{lem:gEmbeddingMainConstruction}
  With $E$, $E'$ and $E_\omega$ as in \fref{dfn:MainConstruction}, the embeddings $E \subseteq E' \subseteq E_\omega$ are $g$-embeddings.
\end{lem}
\begin{proof}
  Every $\varphi \in \Iso_L(E)$ of $E$ extends to an isometry $\varphi_1\colon \xi \mapsto \xi\circ\varphi^{-1}$ of $K_{C,0}(E)$, which, by the universal property of $E[X]$, extends uniquely to an isometry $\varphi_2 \in \Iso_L(E')$.
  The map $\varphi \mapsto \varphi_1$ is continuous as in the case of $K_0(X)$, and the map $\varphi_1 \mapsto \varphi_2$ is continuous since $K_{C,0}(E)$ generates a dense subset of $E'$.
  These are also clearly homomorphisms.
  The result for $E_\omega$ follows by induction.
\end{proof}

\begin{fct}[See Teleman's Theorem, {\cite[Theorem~2.2.1]{Pestov:WhereTo}}]
  \label{fct:Teleman}
  Every Polish group embeds homeomorphically in $\Iso_L(E)$ for some separable Banach space $E$.
\end{fct}
\begin{proof}
  Let $H$ be a Polish group.
  Then it admits a compatible left-invariant distance $d$, which we may assume is bounded by $1$.
  Adjoining a distinguished point $*$ to $H$ with $d(*,h) = 1$ (or even $\half$) for all $h \in H$ we obtain a pointed metric space $H^*$, and let $E = \AEx(H^*,*)$.
  The left action of $h \in H$ on $H^*$, fixing $*$, is an isometry, which extends by the universal property to a linear isometry $h' \in \Iso_L(E)$, and the map $h \mapsto h'$ is a continuous embedding, which is easily seen to be homeomorphic.
  (In other words, $H \subseteq E$ is a $g$-embedding inducing $H \subseteq \Iso(H) \subseteq \Iso_L(E)$.)
\end{proof}

\begin{thm}
  \label{thm:UniversalIsoG}
  The group $\Iso_L(\bG)$ of linear isometries of the Gurarij space is a universal Polish group, that is to say that every other Polish group embeds there homeomorphically.
  Moreover, every separable normed space is $g$-embeddable in $\bG$.
\end{thm}
\begin{proof}
  The moreover part follows immediately from earlier results, and implies the main assertion by \fref{fct:Teleman}.
\end{proof}

\providecommand{\bysame}{\leavevmode\hbox to3em{\hrulefill}\thinspace}


\begin{thebibliography}{Wea99}

\bibitem[{Ben}]{BenYaacov:MetricFraisse}
Itaï \bgroup\scshape{}{Ben Yaacov}\egroup{},
  \href{http://math.univ-lyon1.fr/~begnac/articles/Fraisse.pdf}
  {\emph{Fraïssé limits of metric structures}}, submitted,
  \href{http://arxiv.org/abs/1203.4459}{arXiv:1203.4459}.

\bibitem[BH]{BenYaacov-Henson:Gurarij}
Itaï \bgroup\scshape{}{Ben Yaacov}\egroup{} and C.~Ward
  \bgroup\scshape{}Henson\egroup{}, \emph{Generic orbits and type isolation in
  the {G}urarij space}, research notes.

\bibitem[FJ03]{Fleming-Jamison:IsometriesOnFunctionSpaces}
Richard~J. \bgroup\scshape{}Fleming\egroup{} and James~E.
  \bgroup\scshape{}Jamison\egroup{}, \emph{Isometries on {B}anach spaces:
  function spaces}, Chapman \& Hall/CRC Monographs and Surveys in Pure and
  Applied Mathematics, vol. 129, Chapman \& Hall/CRC, Boca Raton, FL, 2003.

\bibitem[Gur66]{Gurarij:UniversalPlacement}
Vladimir~I. \bgroup\scshape{}Gurarij\egroup{}, \emph{Spaces of universal
  placement, isotropic spaces and a problem of {M}azur on rotations of {B}anach
  spaces}, Sibirsk. Mat. {\v Z}. \textbf{7} (1966), 1002--1013.

\bibitem[Kat88]{Katetov:UniversalMetricSpaces}
Miroslav \bgroup\scshape{}Katětov\egroup{}, \emph{On universal metric spaces},
  General topology and its relations to modern analysis and algebra, {VI}
  ({P}rague, 1986), Res. Exp. Math., vol.~16, Heldermann, Berlin, 1988,
  pp.~323--330.

\bibitem[KS]{Kubis-Solecki:GurariiUniqueness}
Wiesław \bgroup\scshape{}Kubiś\egroup{} and Sławomir
  \bgroup\scshape{}Solecki\egroup{}, \emph{A proof of the uniqueness of the
  {G}urari{\u\i} space}, Israel Journal of Mathematics, to appear,
  \href{http://arxiv.org/abs/1110.0903}{arXiv:1110.0903}.

\bibitem[Lus76]{Lusky:UniqueGurarij}
Wolfgang \bgroup\scshape{}Lusky\egroup{}, \emph{The {G}urarij spaces are
  unique}, Archiv der Mathematik \textbf{27} (1976), no.~6, 627--635.

\bibitem[Pes99]{Pestov:WhereTo}
Vladimir \bgroup\scshape{}Pestov\egroup{}, \emph{Topological groups: where to
  from here?}, Proceedings of the 14th {S}ummer {C}onference on {G}eneral
  {T}opology and its {A}pplications ({B}rookville, {NY}, 1999), vol.~24, 1999,
  pp.~421--502 (2001),
  \href{http://arxiv.org/abs/math/9910144}{arXiv:math/9910144}.

\bibitem[Usp90]{Uspenskij:UrysohnIsometryGroup}
Vladimir~V. \bgroup\scshape{}Uspenskij\egroup{},
  \href{http://www.digizeitschriften.de/dms/img/?PPN=PPN316342866_0031&DMDID=dmdlog28}
  {\emph{On the group of isometries of the {U}rysohn universal metric space}},
  Commentationes Mathematicae Universitatis Carolinae \textbf{31} (1990),
  no.~1, 181--182.

\bibitem[Usp08]{Uspenskij:SubgroupsOfMinimalTopologicalGroups}
\bysame, \emph{On subgroups of minimal topological groups}, Topology and its
  Applications \textbf{155} (2008), no.~14, 1580--1606,
  \href{http://dx.doi.org/10.1016/j.topol.2008.03.001}{doi:10.1016/j.topol.2008.03.001},
  \href{http://arxiv.org/abs/math/0004119}{arXiv:math/0004119}.

\bibitem[Wea99]{Weaver:LipschitzAlgebras}
Nik \bgroup\scshape{}Weaver\egroup{}, \emph{Lipschitz algebras}, World
  Scientific Publishing Co. Inc., River Edge, NJ, 1999.

\end{thebibliography}
\end{document}